\documentclass[11pt]{article}

\usepackage{amsfonts, amsmath, amssymb, amsthm, authblk, hyperref}
\usepackage[usenames]{color}
\usepackage[top=1cm, bottom=2cm, left=2cm, right=2cm]{geometry}

\title{\huge{Determination of elliptic curves by their adjoint $p$-adic $L$-functions}}
\author[$\dagger$]{\Large Maria M. Nastasescu}
\affil[$\dagger$]{\normalsize Department of Mathematics, California Institute of Technology, MC 253-37, Pasadena, CA 91125, USA}
\date{}

\newtheorem{thm}{Theorem}
\newtheorem*{thm*}{Theorem}
\newtheorem{cor}{Corollary}
\newtheorem*{cor*}{Corollary}
\newtheorem{lem}{Lemma}

\newenvironment{remark}[1][Remark]{\begin{trivlist}
\item[\hskip \labelsep {\bfseries #1}]}{\end{trivlist}}

\newcommand\ba{\begin{eqnarray}}
\newcommand\ea{\end{eqnarray}}
\newcommand\nn{\nonumber}

\numberwithin{equation}{section}

\begin{document}
\maketitle

\begin{abstract}
Fix $p$ an odd prime. Let $E$ be an elliptic curve over $\mathbb{Q}$ with semistable reduction at $p$. We show that the adjoint $p$-adic $L$-function of $E$ evaluated at infinitely many integers prime to $p$ completely determines up to a quadratic twist the isogeny class of $E$. To do this, we prove a result on the determination of isobaric representations of $\text{GL}(3, \mathbb{A}_\mathbb{Q})$ by certain $L$-values of $p$-power twists.
\end{abstract}

\section{Introduction}

In this paper we will prove the following result concerning the $p$-adic $L$-function of the symmetric square of an elliptic curve over $\mathbb{Q}$, denoted $L_p (Sym^2 E, s)$ for $ s\in \mathbb{Z}_p$. More specifically, Theorem 1 gives a generalization of the result obtained in \cite{LuoRam} concerning $p$-adic $L$-functions of elliptic curves over $\mathbb{Q}$:

\begin{thm}
Let $p$ be an odd prime and $E, E'$ be elliptic curves over $\mathbb{Q}$ with semistable reduction at $p$. Suppose 
\begin{equation}
\label{condition}
L_p (Sym^2 E, n) = C L_p (Sym^2 E', n)
\end{equation}
for all integers $n$ prime to $p$ in an infinite set $Y$ and some constant $C \in \overline{\mathbb{Q}}$. Then $E'$ is isogenous to a quadratic twist $E_D$ of $E$. If $E, E'$ have square free conductors, then in fact $E \approx E'$ over $\mathbb{Q}$. 
\end{thm} 

Suppose $E$ has good reduction at $p$. Following the definition in \cite{Dabro} of the $p$-adic $L$-function for the symmetric square of an elliptic curve $E$ over $\mathbb{Q}$, if $\chi: \mathbb{Z}_p^\times \to \mathbb{C}_p^\times$ is a wild $p$-adic character of conductor $p^{m_\chi}$, which can be identified with a primitive Dirichlet character, then 
\begin{equation}
\label{connection}
L_p (Sym^2 E, \chi) = \int_{\mathbb{Z}_p^\times} \chi \text{d}\mu_p = C_{E} \cdot \alpha_p^{-2m_\chi} \tau(\overline{\chi})^2 p^{m_\chi} L(Sym^2 E, \chi, 2)
\end{equation}
where $C_E$ is a constant that depends on $E$, $\tau(\chi)$ is the Gauss sum of $\chi$ and $\alpha_p$ is a root of the polynomial $X^2 - a_p X +p$, with $a_p = p+1 - \# E(\mathbb{F}_p)$. It is proved in \cite{Dabro} that if $E$ has good ordinary reduction at $p$ then $\mu_p$ is a bounded measure on $\mathbb{Z}_p^\times$, while if $E$ has good supersingular reduction at $p$ then $\mu_p$ is $h$-admissible (cf. \cite{Visik}) with $h=2$. 

Similarly, if $E$ has bad multiplicative reduction at $p$, then for a non-trivial even character as above we have 
\begin{equation}
\label{connection2}
L_p (Sym^2 E, \chi) = \int_{\mathbb{Z}_p^\times} \chi \text{d}\mu_p = C_E' \tau(\overline{\chi})^2 p^{m_\chi} L(Sym^2 E, \chi, 2),
\end{equation}
with $\mu_p$ bounded on $\mathbb{Z}_p^\times$. 

Set
\begin{equation*}
L_p (Sym^2 E, \chi, s) := L_p (Sym^2 E, \chi \cdot \langle x \rangle^s)
\end{equation*}
where $\langle \cdot \rangle : \mathbb{Z}_p^\times \to 1 + p \mathbb{Z}_p$, with $\langle x \rangle = \frac{x}{\omega (x)}$ and $\omega : \mathbb{Z}_p^\times \to \mathbb{Z}_p^\times$ the Teichm\"{u}ller character. 

Using the theory on $h$-admissible measures developed in \cite{Visik}, identity \eqref{condition} implies that 
\begin{equation*}
L_p( Sym^2 E, \chi, s) = C L_p (Sym^2 E', \chi, s)
\end{equation*}
holds for all $s \in \mathbb{Z}_p$ and $\chi$ a wild $p$-adic character. 

Let $f, f'$ be the newforms of weight 2 associated to $E$ and $E'$, and $\pi, \pi'$ the unitary cuspidal automorphic representations of $\text{GL}(2, \mathbb{A}_\mathbb{Q})$ generated by $f$ and $f'$ respectively. Then 
\begin{equation}
L(Sym^2 E, s) = L(Sym^2 \pi, s-1)
\end{equation}
where $Sym^2 \pi$ is the automorphic representation of $\text{GL}(3, \mathbb{A}_\mathbb{Q})$ associated to $\pi$ by Gelbart and Jacquet in \cite{GelJac}. 

Hence, using \eqref{connection}, Theorem 1 is a consequence of the following result on the determination of isobaric automorphic representations of $\text{GL}(3)$ over $\mathbb{Q}$, which will be proved in Section 4.
\begin{thm}
Suppose $\pi, \pi'$ are two isobaric sums of unitary cuspidal automorphic representations of $\text{GL}(3, \mathbb{A}_\mathbb{Q})$ with the same central character $\omega$. Let $X_{(p)}^w$ be the set of $p$-power order characters of conductor $p^a$ for some $a$. Suppose $L(\pi \otimes \chi, s)$ is entire for all $\chi \in X_{(p)}^w$, and that there exist constants $B,C \in \mathbb{C}$ such that
\begin{equation}
\label{cond_ww}
L(\pi \otimes \chi, \beta) = B^a C L(\pi' \otimes \chi, \beta)
\end{equation}
for some $1\geq \beta > \frac{2}{3}$ and for all $\chi \in X_{(p),a}^w$ primitive $p$-power order characters of conductor $p^a$ for all but a finite number of $a$. Then $\pi \cong \pi'$. Note that if $\pi, \pi'$ are isobaric sums of tempered unitary cuspidal automorphic representations then the same result holds if \eqref{cond_ww} is satisfied for some $1 \geq \beta > \frac{1}{2}$ (If the generalized Ramanujan conjecture is true then this condition is automatically satisfied). 
\end{thm}

Note that in \cite{Munshi}, a result was proved concerning the determination of $\text{GL}(3)$ forms by twists of characters of almost prime modulus of the central $L$-values. In our case, we twist over a more sparse set of characters. 

$Sym^2 \pi$ is cuspidal iff $E$ is non-CM. Hence, if in Theorem 1 $E,E'$ are two non-CM elliptic curves, then it is enough to prove Theorem 2 for $\pi, \pi'$ cuspidal automorphic representations of $\text{GL}(3, \mathbb{A}_\mathbb{Q})$. 

If $E, E'$ have complex multiplication, let $\eta, \eta'$ be the associated (unitary) idele class characters over the imaginary quadratic number fields $K, K'$ and let $\pi, \pi'$ be the cuspidal automorphic representations of $\text{GL}(2, \mathbb{A}_\mathbb{Q})$ automorphically induced by $\eta$ and $\eta'$ respectively. Then $\pi, \pi'$ are dihedral and moreover
\begin{equation*}
Sym^2(\pi) \cong I_K^{\mathbb{Q}} (\eta^2) \boxplus \eta_0
\end{equation*}
with $\eta_0$ the restriction of $\eta$ to $\mathbb{Q}$, and similarly for $Sym^2(\pi')$. Here $\boxplus$ denotes the isobaric sum (see Section 2). Twisting by a character $\chi \neq \eta_0^{-1}$, since $L(Sym^2(\pi) \otimes \chi, s)$ is an entire function, it follows that Theorem 1 is a consequence of Theorem 2. Note that if $K=K'$ then $\eta_0 = \eta_0'$ and Theorem 2 is a direct consequence of Theorem A in \cite{LuoRam}.

Using Theorem 4.1.2 in \cite{Ram}, the following is a consequence of Theorem 2:
\begin{thm}
Suppose $\pi, \pi'$ are two unitary cuspidal automorphic representations of $\text{GL}(2, \mathbb{A}_\mathbb{Q})$ with the same central character $\omega$. Suppose there exist constants $B,C \in \mathbb{C}$ such that
\begin{equation}
\label{adjoint_cw}
L(Ad(\pi) \otimes \chi, \beta) = B^a C L(Ad(\pi') \otimes \chi, \beta)
\end{equation}
for some $1 \geq \beta > \frac{2}{3}$ and for all $\chi \in X_{(p),a}^w$ primitive $p$-power order characters of conductor $p^a$ for all but a finite number of $a$. Then there exists quadratic character $\nu$ such that $\pi \cong \pi' \otimes \nu$. If $\pi, \pi'$ are tempered then the same result holds if $\eqref{adjoint_cw}$ is true for some $1 \geq \beta > \frac{1}{2}$. 
\end{thm}

In Section 3, we will prove the following result on isobaric sums of unitary cuspidal automorphic representations of $\text{GL}(n, \mathbb{A}_\mathbb{Q})$ for $n\geq 3$: 

\begin{thm}
Let $\pi$ be an isobaric sum of unitary cuspidal automorphic representations of $\text{GL}(n, \mathbb{A}_\mathbb{Q})$ with $n\geq 3$ and $s,r$ be integers relatively prime to $p$. If $L(\pi \otimes \chi, s)$ is entire for all $\chi$ $p$-power order characters of conductor $p^a$ for some $a$, then
\begin{equation}
\lim_{a \to \infty} p^{-a} \sum\nolimits_{\chi \text{ mod } p^a}^* \overline{\chi}(s) \chi(r) L(\pi \otimes \chi, \beta) = \frac{1}{p} \left( 1 - \frac{1}{p} \right) \frac{a_\pi (s/r)}{(s/r)^\beta} 
\end{equation}
where $\sum^*$ denotes the sum over primitive $p$-power order characters of conductor $p^a$ and $1\geq \beta > \frac{n-1}{n+1}$ if $\pi$ is an isobaric sum of tempered unitary cuspidal automorphic representations and $1\geq \beta > \frac{n-1}{n}$ in general. 
\end{thm}

This result generalizes Proposition 2.2 in \cite{LuoRam}. To prove Theorem 4, we will check that the approximate functional equation for $L(\pi \otimes \chi, \beta)$ in \cite{Luo} holds for isobaric automorphic representations of $\text{GL}(n, \mathbb{A}_\mathbb{Q})$ if $L(\pi \otimes \chi, s)$ is entire. Note that in \cite{Harcos}, a similar functional equation for the $L$-function associated to an isobaric automorphic representation at the center $\beta = \frac{1}{2}$ was proved. 

Theorem 4, together with the Generalized Strong Multiplicity One Theorem (see Section 2) proves Theorem 2. As a consequence of Theorem 4, the following non-vanishing result holds:
\begin{cor}
\label{corr}
Let $\pi$ be an isobaric sum of unitary cuspidal automorphic representations of $\text{GL}(n, \mathbb{A}_\mathbb{Q})$ with $n\geq 3$. There are infinitely many primitive $p$-power order characters $\chi$ of conductor $p^a$ for some $a$, such that if $L(\pi \otimes \chi, s)$ is entire for all such characters then $L(\pi \otimes \chi, \beta) \neq 0$ for all $\beta \not\in \left[ \frac{2}{n+1}, 1 - \frac{2}{n+1}\right]$ if $\pi$ is an isobaric sum of tempered unitary cuspidal automorphic representations and for $\beta \not\in \left[ \frac{1}{n}, 1- \frac{1}{n} \right]$ in general. 
\end{cor}
A similar nonvanishing result involving $p$-power twists of cuspidal automorphic representations of $\text{GL}(n, \mathbb{A}_\mathbb{Q})$ was proved in \cite{Ward} for $\beta \not\in \left[ \frac{2}{n+1}, 1 - \frac{2}{2n+1}\right]$. In \cite{Barthel} a nonvanishing result for $\beta$ in the same intervals as in Corollary \ref{corr} was proved for all twists of $L$-functions of $\text{GL}(n)$, instead of just for $p$-power twists. In \cite{Luo}, the result in \cite{Barthel} was further improved to the interval $\beta \not\in \left[ \frac{2}{n}, 1- \frac{2}{n}\right]$. Note that the set of primitive characters of $p$-power order of conductor $p^a$ for some $a$ is more sparse than the set of characters considered in \cite{Barthel} and \cite{Luo}. 

We should also note that for $n=2$, Rohrlich \cite{Rohrlich} proves that if $f$ is a newform of weight 2, then for all but finitely many twists by Dirichlet characters, the $L$-function is nonvanishing at $s=1$. 

\bigskip
\textbf{Acknowledgments:} The author would like to thank her advisor Dinakar Ramakrishnan, Wenzhi Luo, Philippe Michel and Gergely Harcos for useful discussions.

\section{Preliminaries}

Let $\pi$ be an irreducible automorphic representation of $\text{GL}(n, \mathbb{A}_\mathbb{Q})$ and $L(\pi, s)$ its associated $L$-function. Write $\pi = \otimes'_v \pi_v$ as a restricted direct product with $\pi_v$ admissible irreducible representations of the local groups $GL(n,\mathbb{Q}_v)$. The Euler product
\begin{equation}
L(\pi, s) = \prod_v L(\pi_v, s)
\end{equation}
converges for $\text{Re}(s)$ large. There exist conjugacy classes of matrices $A_v(\pi) \in \text{GL}(n, \mathbb{C})$ such that the local $L$-functions at finite places $v$ with $\pi_v$ unramified are
\begin{equation}
L(\pi_v, s) = \text{det} ( 1- A_v (\pi) q_v^{-s})^{-1}
\end{equation}
with $q_v$ the order of the residue field at $v$. We can take $A_v (\pi) =[\alpha_{1,v} (\pi), \dotsb , \alpha_{n, v} (\pi)]$ to be diagonal representatives of the conjugacy classes.

For $S$ a set of places of $\mathbb{Q}$ we can define 
\begin{equation}
L^S (\pi,s) = \prod_{v \not\in S} L_v (\pi,s)
\end{equation}
called the incomplete $L$-function associated to set $S$. 

Let $\boxplus$ be the isobaric sum introduced in \cite{JacquetShalika}. We can define an irreducible automorphic representation, called an isobaric representation, $\pi_1 \boxplus \dotsb \boxplus \pi_m$ of $\text{GL}(n, \mathbb{A}_\mathbb{Q})$, $n= \sum_{i=1}^m n_i $, for $m$ cuspidal automorphic representations $\pi_i \in \text{GL}(n_i, \mathbb{A}_\mathbb{Q})$. Such a representation satisfies
\begin{equation*}
L^S (\boxplus_{j=1}^m \pi_j, s) = \prod_{j=1}^m L^S (\pi_j, s) 
\end{equation*}
with $S$ a finite set of places. 

We say that an isobaric representation is tempered if each $\pi_i$ in the isobaric sum $\pi = \pi_1 \boxplus \dotsb \boxplus \pi_m$ is a tempered cuspidal automorphic representation, or more specifically if each local factor $\pi_{i,v}$ is tempered. 

Since we will want bound \eqref{avg} on the coefficients of the Dirichlet series \eqref{diric} to hold, we will consider a subset of the set of isobaric representations of $\text{GL}(n, \mathbb{A}_\mathbb{Q})$, more specifically, those given by an isobaric sum of unitary cuspidal automorphic representations. We denote this subset by $\mathcal{A}_u(n)$. We will also consider the case when the unitary cuspidal automorphic representations in the isobaric sum are tempered, which is expected to always happen if the generalized Ramanujan conjecture is true. 

The following generalization of the Strong Multiplicity One Theorem for isobaric representations is due to Jacquet and Shalika (see \cite{JacquetShalika}): 

\begin{thm*}[Generalized Strong Multiplicity One]
Consider two isobaric representations $\pi_1$ and $\pi_2$ of $\text{GL}(n, \mathbb{A}_\mathbb{Q})$ and $S$ a finite set of places of $\mathbb{Q}$ that contains $\infty$, such that $\pi_1$ and $\pi_2$ are unramified outside set $S$. Then $\pi_{1,v} \cong \pi_{2,v}$ for all $v \not\in S$ implies $\pi_1 \cong \pi_2$.
\end{thm*}

Let $n\geq 3$ and let $\pi \in \mathcal{A}_u(n)$ be an isobaric sum of unitary cuspidal automorphic representations of $\text{GL}(n, \mathbb{A}_{\mathbb{Q}})$ with (unitary) central character $\omega_\pi$ and contragradient representation $\tilde{\pi}$. We have
\begin{equation*}
L(\pi_\infty, s) = \prod_{j=1}^n \pi^{-\frac{s- \mu_j}{2}} \Gamma \left( \frac{s-\mu_j}{2} \right), \ L(\tilde{\pi}_\infty, s) =  \prod_{j=1}^n \pi^{-\frac{s- \overline{\mu_j}}{2}} \Gamma \left( \frac{s-\overline{\mu_j}}{2} \right)
\end{equation*}
for some $\mu_j \in \mathbb{C}$, with $\pi$ in this context denoting the transcendental number.

The $L$-function is defined for $\text{Re}(s) >1$ by the absolutely convergent Dirichlet series
\begin{equation}
\label{diric}
L(\pi, s) = \sum_{m=1}^{\infty}  \frac{a_\pi (m)}{m^{s}}
\end{equation}
with $a_\pi (1) =1$. This extends to a meromorphic function on $\mathbb{C}$ with a finite number of poles. 

It is known that the coefficients $a_\pi (m)$ of the Dirichlet series satisfy
\begin{equation}
\label{avg}
\sum_{m\leq M} |a_\pi (m)|^2 \ll_{\epsilon} M^{1+\epsilon}
\end{equation}
for $M \geq 1$ (cf. Theorem 4 in \cite{Mol}, \cite{Jacquet, JacquetShalika, Shahidi1, Shahidi2}).

The completed $L$-function $\Lambda(\pi, s) = L(\pi_\infty, s) L(\pi, s)$ obeys the functional equation
\begin{equation}
\Lambda(\pi, s) = \epsilon(\pi, s) \Lambda(\tilde{\pi}, 1-s)
\end{equation} 
where the $\epsilon$-factor is given by
\begin{equation}
\epsilon(\pi, s) = f_\pi^{1/2-s} W(\pi)
\end{equation}
and $f_\pi$ and $W(\pi)$ are the conductor and the root number of $\pi$.

Let $\chi$ denote an even primitive Dirichlet character that is unramified at $\infty$ and with odd conductor $q$ coprime to $ f_\pi$. The twisted $L$-function obeys the functional equation
\begin{equation}
\Lambda(\pi \otimes \chi, s) = \epsilon(\pi \otimes \chi, s) \Lambda(\tilde{\pi} \otimes \overline{\chi}, 1-s)
\end{equation}
where $\Lambda(\pi \otimes \chi, s) = L(\pi_\infty, s) L(\pi \otimes \chi, s)$.
The $\epsilon$-factor is given by
\begin{equation}
\epsilon(\pi \otimes \chi, s) = \epsilon(\pi, s) \omega_\pi (q) \chi(f_\pi) q^{-ns} \tau(\chi)^n
\end{equation}
with $\tau(\chi)$ the Gauss sum of the character $\chi$ (cf. Proposition 4.1 in \cite{Barthel}).

Since $L(\pi \otimes \chi, s)$ does not vanish in the half-plane $\text{Re}(s)>1$, it is enough  to consider $1/2 \leq \text{Re}(s) \leq1$. Twisting $\pi$ by a unitary character $|\cdot|^{it}$ if needed, take $s \in \mathbb{R}$. Hence, from now on, 
\begin{equation}
\frac{1}{2} \leq s \leq 1. 
\end{equation}

We now present a construction also introduced in \cite{Luo,LuoRam}. For a smooth function $g$ with compact support on $(0,\infty)$, normalized such that $\int_0^\infty g(u) \frac{du}{u} =1$, we can introduce an entire function $k$ to be
\begin{equation*}
k(s) = \int_0^\infty g(u) u^{s-1} du
\end{equation*}
such that $k(0)=1$ by normalization and $k$ decreases rapidly in vertical strips. We then  define two functions for $y>0$,
\ba
F_1 (y) &=& \frac{1}{2\pi i} \int_{(2)} k(s) y^{-s} \frac{ds}{s}, \\
F_2(y) &=& \frac{1}{2\pi i} \int_{(2)} k(-s) G(-s + \beta) y^{-s} \frac{ds}{s},
\ea
with $G(s) = \frac{L(\tilde{\pi}_\infty, 1-s)}{L(\pi_\infty, s)}$ and the integrals above are over $\text{Re}(s)=2$. 
The functions $F_1(y)$ and $F_2(y)$ obey the following relations (see \cite{Luo}):
\begin{enumerate}
\item $F_{1,2} (y) \ll C_m y^{-m} \text{ for all } m\geq 1$, as $y\to \infty$.
\item $F_1 (y) = 1+ O(y^m) \text{ for all } m\geq 1$ for $y$ small enough. 
\item $F_2(y) \ll_\epsilon 1+ y^{1-\eta - \text{Re}(\beta)-\epsilon}$ for any $\epsilon>0$, where $\eta = \max_{1 \leq j \leq n} \text{Re}(\mu_j)$. The following inequality holds (see \cite{LRS}):
\begin{equation}
0 \leq \eta \leq \frac{1}{2} - \frac{1}{n^2+1}.
\end{equation}
\end{enumerate}

The following approximate functional equation was first proved in \cite{Luo} for cuspidal automorphic representations of $\text{GL}(n)$ over $\mathbb{Q}$. We verify that it holds for $\pi \in \mathcal{A}_u(n)$ such that $L(\pi \otimes \chi, s)$ is entire. A similar approximate functional equation was proved in \cite{Harcos} for $L(\pi, \beta)$ at the center $\beta = \frac{1}{2}$, for slightly different rapidly decreasing functions. 

\begin{thm*}
If $\pi \in \mathcal{A}_u(n)$ and $\chi$ a primitive Dirichlet character of conductor $q$ such that $L(\pi \otimes \chi, s)$ is entire, then for any $\frac{1}{2} \leq \beta \leq 1$
\begin{equation*}
L(\pi \otimes \chi, \beta) = \sum_{m=1}^\infty \frac{a_\pi (m) \chi (m)}{m^\beta} F_1 \left( \frac{my}{f_\pi q^n}\right) + \omega_\pi (q) \epsilon(0, \pi) \tau(\chi)^n (f_\pi q^n)^{-\beta} \sum_{m=1}^\infty \frac{a_{\tilde{\pi}} (m) \overline{\chi} (mf_\pi')}{m^{1-\beta}} F_2 \left(\frac{m}{y}\right),
\end{equation*}
where $f_\pi'$ is the multiplicative inverse of $f_\pi$ modulo $q$.
\end{thm*}

\begin{proof}
For $\sigma > 0,\ y >0$ consider the integral:
\begin{equation*}
\frac{1}{2\pi i} \int_{(\sigma)} k(s) L(\pi \otimes \chi, s+\beta) \left( \frac{y}{f_\pi q^{n}} \right)^{-s} \frac{\text{d}s}{s}.
\end{equation*}
Since $k(s)$ and $L(\pi \otimes \chi, s+\beta)$ are entire functions, the only pole of the function 
$$k(s) L(\pi \otimes \chi, s+\beta) \left( \frac{y}{f_\pi q^{n}} \right)^{-s} s^{-1}$$
is a simple pole at $s=0$ with residue equal to
$$
\text{lim}_{s\to 0} k(s) L(\pi \otimes \chi, s+\beta) \left( \frac{y}{f_\pi q^{n}} \right)^{-s} = L(\pi \otimes \chi, \beta)
$$
since $k(0)=1$. Then by the residue theorem 
$$
L(\pi \otimes \chi, \beta) = \frac{1}{2\pi i} \int_{(\sigma)} k(s) L(\pi \otimes \chi, s+\beta) \left( \frac{y}{f_\pi q^{n}} \right)^{-s} \frac{\text{d}s}{s} - \frac{1}{2\pi i} \int_{(-\sigma)} k(s) L(\pi \otimes \chi, s+\beta) \left( \frac{y}{f_\pi q^{n}} \right)^{-s} \frac{\text{d}s}{s}.
$$
Taking $s \to -s$ in the second integral gives
$$
L(\pi \otimes \chi, \beta) = \frac{1}{2\pi i} \int_{(\sigma)} k(s) L(\pi \otimes \chi, s+\beta) \left( \frac{y}{f_\pi q^{n}} \right)^{-s} \frac{\text{d}s}{s} + \frac{1}{2\pi i} \int_{(\sigma)} k(-s) L(\pi \otimes \chi, -s+\beta) \left( \frac{y}{f_\pi q^{n}} \right)^{s} \frac{\text{d}s}{s}.
$$
The functional equation is
$$
L(\pi_\infty, s) L(\pi \otimes \chi, s) = \epsilon(\pi \otimes \chi, s) L(\tilde{\pi}_\infty, 1-s) L(\tilde{\pi} \otimes \overline{\chi}, 1- s),
$$
which implies that 
$$
L(\pi \otimes \chi, s) = \epsilon(\pi \otimes \chi, s) G(s) L(\tilde{\pi} \otimes \overline{\chi}, 1-s). 
$$
Substituting this identity in the second integral gives 
\begin{equation}
\label{sumof2i}
L(\pi \otimes \chi, \beta) = I_1 + I_2
\end{equation}
with 
$$
I_1 = \frac{1}{2\pi i} \int_{(\sigma)} k(s) L(\pi \otimes \chi, s+\beta) \left( \frac{y}{f_\pi q^{n}} \right)^{-s} \frac{\text{d}s}{s}
$$
and
$$
I_2 = \frac{1}{2\pi i} \int_{(\sigma)} \epsilon(\beta - s, \pi \otimes \chi) G(\beta - s) k(-s) L(\tilde{\pi} \otimes \overline{\chi}, 1 + s - \beta) \left( \frac{y}{f_\pi q^n} \right)^s \frac{\text{d}s}{s}.
$$
Taking $\sigma = 2$ and substituting with $L(\pi \otimes \chi, s) = \sum_{m=1}^\infty a_\pi (m) \chi(m) m^{-s}$ in the region of absolute convergence gives
$$
I_1 = \sum_{m=1}^\infty a_\pi (m) \chi(m) m^{-\beta} \cdot \frac{1}{2\pi i} \int_{(2)} k(s) \left( \frac{my}{f_\pi q^n} \right)^{-s} \frac{\text{d}s}{s}, 
$$ 
and by the definition of $F_1$,
\begin{equation}
\label{i1}
I_1 = \sum_{m=1}^\infty a_\pi (m) \chi(m) m^{-\beta} F_1 \left( \frac{my}{f_\pi q^n} \right).
\end{equation}
Similarly,
$$
I_2 = \frac{1}{2\pi i} \int_{(2)} \epsilon(\beta - s, \pi \otimes \chi) G(\beta -s) k(-s) \sum_{m=1}^\infty a_{\tilde{\pi}} (m) \overline{\chi}(m) m^{-1-s+\beta} \left( \frac{y}{f_\pi q^n} \right)^s \frac{\text{d}s}{s}
$$
with $\epsilon(\beta - s, \pi \otimes \chi) = \epsilon(\beta -s, \pi) \omega_\pi (q) \chi(f_\pi) q^{-n(\beta -s)} \tau(\chi)^n$ and $\epsilon(\beta -s, \pi) = f_\pi^{1/2 - \beta +s} W(\pi)$.
This gives 
$$
I_2 = \sum_{m=1}^\infty a_{\tilde{\pi}} (m) \overline{\chi} (m f_\pi') m^{-1 + \beta} f_\pi^{1/2-\beta} W(\pi) \omega_\pi (q) q^{-n \beta} \tau(\chi)^n \cdot \frac{1}{2\pi i} \int_{(2)} G(\beta -s) k(-s) y^s m^{-s} \frac{\text{d}s}{s}. 
$$
By the definition of $F_2$,
\begin{equation}
\label{i2}
I_2 = \omega_{\pi} (q) \epsilon(0, \pi) \tau(\chi)^{n} (f_\pi q^n)^{-\beta} \sum_{m=1}^\infty \frac{a_{\tilde{\pi}} (m) \overline{\chi} (m f_\pi')}{m^{1-\beta}} F_2 \left( \frac{m}{y} \right). 
\end{equation}
Here $W(\pi) f_\pi^{1/2} = \epsilon(0, \pi)$. Applying equations \eqref{sumof2i}, \eqref{i1} and \eqref{i2} gives the desired approximate functional equation.
\end{proof}

For an odd prime $p$, define the sets (following the notations in \cite{LuoRam}):
\begin{equation*}
X_{(p)} =\{ \chi \text{ a Dirichlet character of conductor } p^a \text{ for some } a \},
\end{equation*}
\begin{equation*}
X_{(p)}^w = \{ \chi \in X_{(p)} | \chi \text{ has } p \text{-power order}\}.
\end{equation*}
The characters of $X_{(p)}^w$ are called wild at $p$. 

If $\chi \in X_{(p)}$, then $\chi: (\mathbb{Z}/ p^a \mathbb{Z})^\times \to \mathbb{C}^\times$ for some $a$. Note that $(\mathbb{Z}/p^a \mathbb{Z})^\times \cong \mathbb{Z}/p^{a-1} \mathbb{Z} \times \mathbb{Z}/(p-1)\mathbb{Z}$. A character in $X_{(p)}$ is an element in $X_{(p)}^w$ if and only if it is trivial on the elements of exponent $p-1$. We denote the integers (mod $p^a$) of exponent $p-1$ by $S_a$ and the sum over all primitive wild characters of conductor $p^a$ by $\sum_{\chi \text{ mod } p^a}^*$. 

Consider the set
\begin{equation}
\label{gpa}
G(p^a):=\text{ker}((\mathbb{Z}/p^a \mathbb{Z})^\times \to (\mathbb{Z}/p)^\times) \cong \mathbb{Z}/p^{a-1} \mathbb{Z}.
\end{equation}

Using the orthogonality of characters we get that summing over the primitive wild characters of conductor $p^a$ gives (see \cite{LuoRam}):
\begin{equation}
\label{ortho}
\sum\nolimits_{\chi \text{ mod } p^a}^* \chi = |G(p^a)| \delta_{S_a} - |G(p^{a-1})| \delta_{S_{a-1}},
\end{equation}
with $|G(p^a)| = p^{a-1}$ from \eqref{gpa} and $\delta_{S_a}$ the characteristic function of $S_a$. 

The following result for hyper-Kloosterman sums was proved in \cite{Yangbo}:
\begin{lem}
\label{kl}
Let $p$ be a prime number, $1<n<p$ and $q=p^a$ with $a>1$. Let $x'$ denote the inverse of $x$ mod $q$ and let $e(x):=e^{2\pi i x}$. Then for any integer $z$ coprime to $p$ the hyper-Kloosterman sum
\begin{equation*}
\Bigl| \sum_{\substack{x_1, \dotsb, x_n (\text{mod } q) \\ (x_i,p)=1}} e \left( \frac{x_1+\dotsb + x_n + z x_1' \dotsb x_n'}{q} \right) \Bigr|
\end{equation*}
is bounded by 
\begin{equation}
\begin{cases} \leq (n+1)q^{n/2} &\mbox{if } 1<n<p-1, a>1 \\
\leq p^{1/2} q^{n/2} & \mbox{if } n=p-1, a\geq 5 \\
\leq pq^{n/2} & \mbox{if } n=p-1, a=4  \\
\leq p^{1/2} q^{n/2} & \mbox{if } n=p-1, a=3 \\
\leq q^{n/2} & \mbox{if } n=p-1, a=2. 
\end{cases}
\end{equation}
\end{lem}
A consequence of Lemma \ref{kl} is the following result:
\begin{lem}
\label{gauss}
Let $\tau(\chi)$ denote the Gauss sum of the character $\chi$. If $(r,p)=1$, then the following bound holds:
\begin{equation}
\Bigl|\sum\nolimits_{\chi \text{ mod }p^a}^* \overline{\chi}(r) \tau^n(\chi) \Bigr| \ll p^{1/2 + a(n+1)/2} 
\end{equation}
for $2 < n \leq p$. 
\end{lem}
\begin{proof}
If $\chi$ is a primitive character of conductor $p^a$, then 
\begin{equation*}
\tau(\chi) = \sum_{m=0}^{p^a-1} \chi(m) e^{2\pi i m/p^a}. 
\end{equation*}
Let 
\begin{equation*}
A:= \sum\nolimits_{\chi \text{ mod }p^a}^* \overline{\chi}(r) \tau^n(\chi),
\end{equation*}
then 
\begin{equation*}
A = \sum\nolimits_{\chi \text{ mod }p^a}^* \left[ \overline{\chi}(r) \left( \sum_{m=0}^{p^a-1} \chi(m) e^{2\pi i m/p^a} \right)^n \right].
\end{equation*}
Rewrite the above sum as follows:
\begin{equation*}
A = \sum\nolimits_{\chi \text{ mod }p^a}^* \left[ \overline{\chi}(r) \left(\sum_{x_1=0}^{p^a-1} \chi(x_1) e^{2\pi i x_1/p^a}\right) \dotsb \left(\sum_{x_n=0}^{p^a-1} \chi(x_n) e^{2\pi i x_n/p^a}\right) \right].
\end{equation*}
This in turn gives 
\begin{equation*}
A= \sum_{x_1=0}^{p^a-1} \dotsb \sum_{x_n=0}^{p^a-1} \sum\nolimits_{\chi \text{ mod }p^a}^* \chi(r') \chi(x_1) \dotsb  \chi(x_n) e \left( \frac{x_1 + \dotsb + x_n}{p^a} \right).
\end{equation*}
Hence,
\begin{equation*}
A= \sum_{x_1=0}^{p^a-1} \dotsb \sum_{x_n=0}^{p^a-1} \left[ \sum\nolimits_{\chi \text{ mod }p^a}^* \chi(r' x_1 \dotsb  x_n) \right] e \left( \frac{x_1 + \dotsb + x_n}{p^a} \right)
\end{equation*}
which by equation \eqref{ortho} gives
\begin{equation*}
A= \sum_{x_1=0}^{p^a-1} \dotsb \sum_{x_n=0}^{p^a-1} e \left( \frac{x_1 + \dotsb + x_n}{p^a} \right) (p^{a-1} \delta_{S_a} ( r' x_1 \dotsb x_n) - p^{a-2} \delta_{S_{a-1}} (r' x_1 \dotsb x_n)).
\end{equation*}
Thus,
\begin{equation}
\label{l1}
A= p^{a-1} \sum_{b \in S_a} T(br, p^a)  - p^{a-2} \sum_{c \in S_{a-1}} \sum_{i=0}^{p-1} T(cr + i p^{a-1}, p^a) 
\end{equation}
where 
\begin{equation*}
T(u, p^a) = \sum_{\substack{x_1, \dotsb , x_{n-1} (\text{mod }p^a) \\ (x_i, p)=1}} e \left( \frac{x_1 + \dotsb + x_{n-1} + u x_1' \dotsb x_{n-1}'}{p^a} \right).
\end{equation*}
From Lemma \ref{kl}, for $(u,p) =1$ and $a$ sufficiently large
\begin{equation}
\label{l2}
|T(u,p^a)| \ll p^{1/2 + a(n-1)/2}.
\end{equation}
From \eqref{l1} and \eqref{l2} it follows that
\begin{equation*}
|A| \ll p^{a-1} (p-1) p^{1/2 + a(n-1)/2} + p^{a-2} (p-1)^2 p^{1/2 + a(n-1)/2}.
\end{equation*}
Thus $|A| \ll p^a p^{1/2 + a(n-1)/2}$.
\end{proof}

\section{Non-vanishing of $p$-power twists on $\text{GL}(n, \mathbb{A}_\mathbb{Q})$}

Let $s,r$ be integers relatively prime to $p$.  For $\pi$ an isobaric sum of unitary cuspidal automorphic representations of $\text{GL}(n, \mathbb{A}_\mathbb{Q})$ define
\begin{equation}
S_{s/r} (p^a, \pi , \beta ) = p^{-a} \sum\nolimits_{\chi \text{ mod } p^a}^* \overline{\chi}(s) \chi(r) L(\pi \otimes \chi, \beta)
\end{equation}
where $\sum^*$ denotes the sum over primitive wild characters of conductor $p^a$. 

In this section we prove Theorem 4, which states that:
\begin{equation}
\lim_{a \to \infty} S_{s/r} (p^a, \pi, \beta) = \frac{1}{p} \left( 1 - \frac{1}{p} \right) \frac{a_\pi (s/r)}{(s/r)^\beta}
\end{equation}
for $\beta > \frac{n-1}{n+1}$ if $\pi$ is tempered, and for $\beta > \frac{n-1}{n}$ in general. Note that in this section, by $\pi$ tempered we will mean an isobaric sum of tempered (unitary) cuspidal automorphic representations. If $s \!\! \not| r$ above, then we define $a_\pi(s/r)$ to be zero.

\begin{proof}[Proof of Theorem 4]

We generalize the proof of Proposition 2.2 in \cite{LuoRam} and use methods also developed in \cite{Luo,Ward}. The following approximate functional equation holds (see Section 2):
\begin{equation*}
L(\pi \otimes \chi, \beta) = \sum_{m=1}^\infty \frac{a_\pi (m) \chi (m)}{m^\beta} F_1 \left( \frac{my}{f_\pi p^{an}}\right) + \omega_\pi (p^a) \epsilon(0, \pi) \tau(\chi)^n (f_\pi p^{an})^{-\beta} \sum_{m=1}^\infty \frac{a_{\tilde{\pi}} (m) \overline{\chi} (mf_\pi')}{m^{1-\beta}} F_2 \left(\frac{m}{y}\right),
\end{equation*}
where $\chi$ is a character of conductor $p^a$ and $f_\pi'$ is the multiplicative inverse of $f_\pi$ modulo $p^a$. 

Define $x$ such that $xy = p^{an}$. Write 
\begin{equation}
S_{s/r}(p^a, \beta) = S_{1,s/r}(p^a, \beta) + S_{2, s/r} (p^a, \beta),
\end{equation}
where 
\begin{equation}
S_{1, s/r} (p^a, \beta) = p^{-a}  \sum\nolimits_{\chi \text{ mod }p^a}^* \sum_{m=1}^\infty \frac{a_\pi (m) \chi(ms'r)}{m^\beta} F_1 \left( \frac{m}{f_\pi x} \right)
\end{equation}
and 
\begin{equation}
\label{s2sr}
S_{2,s/r} (p^a, \beta) = p^{-a} \omega_\pi (p^a) \sum\nolimits_{\chi \text{ mod }p^a}^* \epsilon(0, \pi) \tau(\chi)^n (f_\pi p^{an})^{-\beta} \sum_{m=1}^\infty \frac{a_{\tilde{\pi}} (m) \overline{\chi} (m s' r f_\pi')}{m^{1-\beta}} F_2 \left(\frac{m}{y}\right).
\end{equation}

Let 
\begin{equation}
\label{phisr}
Z_{s/r} (p^a, \beta) = \sum_{b \in S_a} \sum_{\substack{rm \equiv bs (p^a) \\ m \geq 1}} \frac{a_\pi (m)}{m^\beta} F_1 \left( \frac{m}{f_\pi x} \right).
\end{equation}
Then applying equation \eqref{ortho} gives
\begin{equation*}
S_{1, s/r} (p^a) = p^{-a} \sum_{m=1}^\infty \frac{a_\pi (m)}{m^\beta} F_1 \left( \frac{m}{f_\pi x} \right) \left[ p^{a-1} \delta_{S_a} (ms' r) - p^{a-2} \delta_{S_{a-1}} (ms' r) \right],
\end{equation*}
hence 
\begin{equation}
S_{1,s/r} = \frac{1}{p}\left[ Z_{s/r} (p^a, \beta) - p^{-1} Z_{s/r} (p^{a-1}, \beta) \right].
\end{equation}


Assume $r|s$. First, consider the term in \eqref{phisr} with $b=1$ and $m = s/r$. This is a solution to the equation $rm \equiv bs (\text{mod } p^a)$ for all $a$. We will want to set the necessary condition for this to be the dominant contribution. Now if $m \neq s/r$, then $m = bs/r + kp^a$. If $k=0$ then $b \neq 1$ and since $b \in S_{a}$, it follows that $b \gg p^{a/(p-1)}$ which implies
\begin{equation*}
m \gg p^{a/(p-1)}.
\end{equation*}
If $k \neq 0$, then $m \ll k p^a$. 

Decompose
\begin{equation*}
Z_{s/r} (p^a, \beta) = \Sigma_{1,a} + \Sigma_{2,a},
\end{equation*}
where 
\begin{equation}
\Sigma_{1,a} =  \frac{a_\pi (s/r)}{(s/r)^\beta} F_1 \left( \frac{s}{r f_\pi x} \right)
\end{equation}
and 
\begin{equation}
\Sigma_{2,a} = \sum_{b \in S_a} \sum_{\substack{rm \equiv bs (p^a) \\ m \geq 1, m \neq s/r}} \frac{a_\pi (m)}{m^\beta} F_1 \left( \frac{m}{f_\pi x} \right).
\end{equation}
Since $F_1 \left( \frac{m}{f_\pi x} \right) = 1 + O \left( \frac{m}{f_\pi x} \right)$, 
\begin{equation}
\Sigma_{1,a} = \frac{a_\pi (s/r)}{(s/r)^\beta} \left( 1+ O \left( \frac{1}{x} \right) \right).
\end{equation}

Following \cite{LuoRam}, let 
\begin{equation}
b_{m,a} :=
\begin{cases} 
1 &\mbox{if } m = bs/r + k p^a\\
0 & \mbox{otherwise}. 
\end{cases}
\end{equation}
Then
\begin{equation}
\label{sumof2}
\Sigma_{2,a} \ll \Bigl| \sum_{\substack{1 \leq m \ll x^{1+\epsilon} \\ m \neq s/r}} \frac{a_\pi (m)}{m^\beta} b_{m,a} F_1 \left( \frac{m}{f_\pi x} \right) \Bigr| + \Bigl| \sum_{\substack{m \gg x^{1+\epsilon} \\ m \neq s/r}} \frac{a_\pi (m)}{m^\beta} b_{m,a} F_1 \left( \frac{m}{f_\pi x} \right) \Bigr|.
\end{equation}
Define 
\begin{equation*}
P_{2,a} =  \Bigl| \sum_{\substack{1 \leq m \ll x^{1+\epsilon} \\ m \neq s/r}} \frac{a_\pi (m)}{m^\beta} b_{m,a} F_1 \left( \frac{m}{f_\pi x} \right) \Bigr|
\end{equation*}
and
\begin{equation*}
Q_{2,a} = \Bigl| \sum_{\substack{m \gg x^{1+\epsilon} \\ m \neq s/r}} \frac{a_\pi (m)}{m^\beta} b_{m,a} F_1 \left( \frac{m}{f_\pi x} \right) \Bigr|.
\end{equation*}
Since $F_1 \left( \frac{m}{f_\pi x} \right) = 1 + O(x^\epsilon)$ for $m \ll x^{1+\epsilon}$, this gives 
\begin{equation}
\label{before}
P_{2,a} \ll x^\epsilon \Bigl| \sum_{\substack{1 \leq m \ll x^{1+\epsilon} \\ m \neq s/r}} \frac{a_\pi (m)}{m^\beta} b_{m,a}\Bigr|
\end{equation}
and since $F_1 \left( \frac{m}{f_\pi x} \right) \ll \frac{x^t}{m^t}$ for any integer $t$ and $m\gg x^{1+\epsilon}$,
\begin{equation}
\label{greater}
Q_{2,a} \ll x^t \Bigl| \sum_{\substack{m \gg x^{1+\epsilon}\\ m \neq s/r}} \frac{a_\pi (m)}{m^{\beta + t}} b_{m,a} \Bigr|.
\end{equation}

 If $\pi$ is tempered then from \eqref{before}
\begin{eqnarray}
P_{2,a} &\ll& x^\epsilon \sum_{\substack{1 \leq m \ll x^{1+\epsilon}\\ m \neq s/r}} m^{\epsilon - \beta} b_{m,a} \nonumber \\
&\ll& x^\epsilon \left( \sum_{k p^a \ll x^{1+\epsilon}} (kp^a)^{\epsilon - \beta} + (p^{a/(p-1)})^{\epsilon - \beta} \right) \nonumber \\
&\ll& x^\epsilon p^{a(\epsilon - \beta)} \sum_{k \ll \frac{x^{1+\epsilon}}{p^a}} k^{\epsilon - \beta} \nonumber \\
&\ll&  x^\epsilon p^{a(\epsilon - \beta)} \left( \frac{x^{1+\epsilon}}{p^a} \right)^{\epsilon - \beta +1} \nonumber \\
\label{pp}
&\ll& p^{-a} x^{1-\beta +\epsilon}.
\end{eqnarray}
Similarly, from \eqref{greater}
\begin{eqnarray}
Q_{2,a} &\ll& x^t \sum_{m \gg x^{1+\epsilon}} m^{\epsilon- \beta - t} b_{m,a} \nonumber \\
&\ll& x^t \sum_{k \gg \frac{x^{1+\epsilon}}{p^a}} (kp^a)^{\epsilon- \beta - t} \nonumber \\
&\ll& x^t p^{a(\epsilon - \beta - t)} \sum_{k \gg \frac{x^{1+\epsilon}}{p^a}} k^{\epsilon - \beta - t} \nonumber \\
&\ll& x^t p^{a(\epsilon - \beta - t)} \frac{x^{\epsilon - \beta - t + 1}}{p^{\epsilon - \beta - t +1}} \nonumber \\
\label{qq}
&\ll& p^{-a} x^{1 - \beta + \epsilon}.
\end{eqnarray}
From \eqref{sumof2}, \eqref{pp} and \eqref{qq}, if $\pi$ is tempered then
\begin{equation}
\Sigma_{2,a} \ll p^{-a} x^{1-\beta +\epsilon}.
\end{equation}
We want $\Sigma_{2,a} \to 0$ as $a \to \infty$. Substituting with $x=p^{an(1-\upsilon)}$ gives the condition
\begin{equation*}
-1 + (1-\upsilon)n(1-\beta + \epsilon) <0,
\end{equation*}
or equivalently
\begin{equation}
\label{s1_1}
\upsilon > 1 - \frac{1}{n(1-\beta +\epsilon)}.
\end{equation}

If $\pi$ is not tempered, applying Cauchy's inequality to equation \eqref{before} gives
\begin{equation*}
P_{2,a} \ll x^\epsilon \left( \sum_{1 \leq m \ll x^{1+\epsilon}} \frac{|a_\pi (m)|^2}{m^{2\beta}} \right)^{1/2} \cdot \left( \sum_{1 \leq m \ll x^{1+\epsilon}} b_{m,a}^2 \right)^{1/2}
\end{equation*}
hence
\begin{equation*}
P_{2,a} \ll \frac{x^{1/2 + \epsilon}}{p^{a/2}} \left( \sum_{1 \leq m \ll x^{1+\epsilon}} \frac{|a_\pi (m)|^2}{m^{2\beta}} \right)^{1/2}.
\end{equation*}
Split the sum over $m$ in dyadic intervals
\begin{equation*}
P_{2,a} \ll \frac{x^{1/2 + \epsilon}}{p^{a/2}} \left( \sum_{1 \leq i \ll (1+\epsilon) \log (x)} \sum_{m= 2^{i-1} + 1}^{2^i} \frac{|a_\pi (m)|^2}{m^{2\beta}} \right)^{1/2}.
\end{equation*}
By applying inequality \eqref{avg},
\begin{equation*}
P_{2,a} \ll \frac{x^{1/2 + \epsilon}}{p^{a/2}} \left( \sum_{1 \leq i \ll (1+\epsilon) \log (x)} \frac{2^{i+\epsilon}}{2^{2(i-1)\beta}} \right)^{1/2}
\end{equation*}
which gives
\begin{equation}
\label{pp2}
P_{2,a} \ll p^{-a/2} x^{1-\beta + \epsilon}.
\end{equation}

In equation \eqref{greater}, write $t= t_1 + t_2$, with $t_1, t_2$ large integers, and apply Cauchy's inequality:
\ba
Q_{2,a} &\ll& x^{t_1 + t_2} \left( \sum_{m \gg x^{1+\epsilon}} \frac{|a_\pi(m)|^2}{m^{2\beta+2 t_1}} \right)^{1/2} \left( \sum_{m \gg x^{1+\epsilon}} \frac{b_{m,a}^2}{m^{2t_2}} \right)^{1/2} \nn\\
&\ll& x^{t_1 + t_2} \left( \sum_{i \gg (1+\epsilon) \log (x)}  \sum_{2^{i-1} < m \leq 2^i} \frac{|a_\pi (m)|^2}{m^{2\beta + 2t_1}} \right)^{1/2} \left( \sum_{k \gg \frac{x^{1+\epsilon}}{p^a}} \frac{1}{(kp^a)^{2t_2}} \right)^{1/2} .
\ea
Using \eqref{avg}, 
\ba
Q_{2,a} &\ll& x^{t_1+ t_2} \left( \sum_{i \gg (1+\epsilon) \log (x)}  \frac{2^{i+\epsilon}}{2^{(i-1)(2\beta + 2t_1)}} \right)^{1/2} \left( \sum_{k \gg \frac{x^{1+\epsilon}}{p^a}} p^{-2at_2} \frac{1}{k^{2t_2}} \right)^{1/2} \nn\\
&\ll& x^{t_1+ t_2} \left(x^{1-2t_1 - 2\beta + \epsilon} \right)^{1/2} \left( p^{-2at_2} x^{1 -2t_2 +\epsilon} \right)^{1/2} \nn \\
\label{qq2} 
&\ll& p^{-at_2} x^{1-\beta+\epsilon} .
\ea

Putting \eqref{sumof2}, \eqref{pp2} and \eqref{qq2} together gives
\begin{equation}
\Sigma_{2,a} \ll p^{-a/2} x^{1- \beta + \epsilon}.
\end{equation}

Since we want $\Sigma_{2,a} \to 0$, we get the condition
\begin{equation*}
an (1-\upsilon) (1-\beta + \epsilon) < \frac{a}{2}.
\end{equation*}
This gives
\begin{equation}
\label{s1_2}
\upsilon > 1 - \frac{1}{2n(1-\beta +\epsilon)}.
\end{equation}
For $\upsilon$ as above, 
\begin{equation*}
\lim_{a \to \infty} S_{1,s/r} (p^a, \beta) = \frac{1}{p} \left[ \frac{a_\pi (s/r)}{(s/r)^\beta} - \frac{1}{p} \cdot \frac{a_\pi (s/r)}{(s/r)^\beta} \right],
\end{equation*}
hence 
\begin{equation}
\lim_{a \to \infty} S_{1,s/r} (p^a, \beta) = \frac{p-1}{p^2} \cdot \frac{a_\pi (s/r)}{(s/r)^\beta}.
\end{equation}



In \eqref{s2sr} write
\begin{equation}
\label{ssum}
|S_{2,s/r}| \ll A_{2,s/r} + B_{2,s/r},
\end{equation}
where 
\begin{equation}
A_{2,s/r} = p^{-a} p^{-an\beta} \sum_{m\ll y^{1+\epsilon}} \left[ \frac{|a_{\tilde{\pi}} (m)|}{m^{1-\beta}}  F_2 \left( \frac{m}{y} \right) \Bigl| \sum\nolimits_{\chi (\text{mod }p^a)}^* \overline{\chi} (ms'r f_\pi') \tau^n (\chi) \Bigr| \right]
\end{equation}
and 
\begin{equation}
\label{b2}
B_{2,s/r} = p^{-a} p^{-an\beta} \sum_{m\gg y^{1+\epsilon}} \left[ \frac{|a_{\tilde{\pi}} (m)|}{m^{1-\beta}}  F_2 \left( \frac{m}{y} \right) \Bigl| \sum\nolimits_{\chi (\text{mod }p^a)}^* \overline{\chi} (ms'r f_\pi') \tau^n (\chi) \Bigr| \right].
\end{equation}

If $\pi$ is tempered then $|a_{\tilde{\pi}}(m)| \ll m^\epsilon$. Also, $F_2 \left( \frac{m}{y} \right) \ll 1 + \left( \frac{m}{y} \right)^{1-\beta - \epsilon}$ if $m \ll y^{1+\epsilon}$, which gives $F_2 \left( \frac{m}{y} \right) \ll y^{\epsilon(1-\beta)}$. Applying Lemma \ref{gauss}, 
\begin{equation*}
|A_{2,s/r}| \ll p^{-a} p^{-an\beta} p^{1/2 + a(n+1)/2} y^{\epsilon(1-\beta)} \sum_{m=1}^{y^{1+\epsilon}} m^{\epsilon + \beta - 1}
\end{equation*}
and since $\sum_{m=1}^{y^{1+\epsilon}} m^{\epsilon + \beta - 1} \ll y^{(1+\epsilon)(\epsilon + \beta)}$, 
\begin{equation}
\label{withram}
|A_{2,s/r}| \ll p^{-an\beta + a(n-1)/2} y^{\epsilon + \beta},
\end{equation}
for any $\epsilon >0$. 

Assume now $\pi$ is not tempered. By Cauchy's inequality and
\begin{equation*}
\sqrt{1+ (m/y)^2} \ll y^\epsilon \text{ for } m \ll y^{1+\epsilon}
\end{equation*}
we obtain
\begin{equation*}
|A_{2, s/r}| \ll p^{-a} p^{-an\beta} y^\epsilon \left( \sum_{m \ll y^{1+\epsilon}} \frac{|a_{\tilde{\pi}} (m)|^2}{m^{2-2\beta}} \right)^{1/2} \left( \sum_{m=-\infty}^\infty H \left( \frac{m}{y} \right) \Bigl| \sum\nolimits_{\chi \text{ mod } p^a}^* \overline{\chi} (ms'r f_\pi') \tau^n (\chi) \Bigr|^2 \right)^{1/2},
\end{equation*}
where 
\begin{equation*}
H(u):= \frac{1}{\pi (1+ u^2)}.
\end{equation*}

First, consider the term
\begin{equation}
C:= \sum_{m \ll y^{1+\epsilon}} \frac{|a_{\tilde{\pi}} (m)|^2}{m^{2-2\beta}}
\end{equation}
and split the sum in dyadic intervals
\begin{equation*}
C = \sum_{1 \leq i \ll (1+\epsilon) \log (y)} \sum_{m=2^{i-1} +1}^{2^i} \frac{|a_{\tilde{\pi}} (m)|^2}{m^{2-2\beta}}.
\end{equation*}
Applying \eqref{avg} it follows that
\begin{equation*}
C \ll \sum_{1 \leq i \ll (1+\epsilon) y} \frac{2^{i+\epsilon}}{2^{(i-1)(2-2\beta)}},
\end{equation*}
which implies that 
\begin{equation}
C \ll y^{2\beta  - 1 + \epsilon}.
\end{equation}
Hence,
\begin{equation}
\label{a2srwoD}
|A_{2, s/r}| \ll y^{\beta - 1/2 + \epsilon} p^{-a-an\beta} \left( \sum_{m=-\infty}^\infty H \left( \frac{m}{y} \right) \Bigl| \sum\nolimits_{\chi \text{ mod } p^a}^* \overline{\chi} (ms'r f_\pi') \tau^n (\chi) \Bigr|^2 \right)^{1/2}.
\end{equation}

Let's now look at 
\begin{equation}
D:= \sum_{m= -\infty}^\infty H \left( \frac{m}{y} \right) \Bigl| \sum\nolimits_{\chi \text{ mod } p^a}^* \overline{\chi} (ms'r f_\pi') \tau^n (\chi) \Bigr|^2. 
\end{equation}
We have 
\begin{equation*}
D = \sum\nolimits_{\chi \text{ mod } p^a}^* \sum\nolimits_{\psi \text{ mod } p^a}^* \Bigl| \tau^n(\chi) \tau^n(\overline{\psi}) \sum_{m=-\infty}^\infty \overline{\chi} \psi (ms'r f_\pi') H\left( \frac{m}{y} \right) \Bigr|
\end{equation*}
Following the general approach of \cite{Luo,Ward}, we consider the diagonal and off-diagonal contributions separately. Let's first compute the terms corresponding to $\chi = \psi$:
\begin{equation*}
\sum\nolimits_{\chi \text{ mod } p^a}^* \Bigl| \tau^n(\chi) \tau^n(\overline{\chi}) \sum_{m=-\infty}^\infty H\left( \frac{m}{y} \right) \Bigr| \ll p^{a+na} \sum_{m = -\infty}^\infty H \left( \frac{m}{y} \right)
\end{equation*}
since there are $\ll p^a$ primitive $p$-power characters and since $|\tau^n (\chi)| = p^{an/2}$ from the properties of the Gauss sum of a primitive character. Use the Fourier transform property that $\mathcal{F}\{g(xA)\} = \frac{1}{A} \hat{g} \left( \frac{\nu}{A} \right)$ for $A>0$ (see also \cite{Luo,Ward}) to get that
\begin{equation*}
\sum_{m = -\infty}^\infty H\left(\frac{m}{y}\right) =y\sum_{\nu =-\infty}^\infty T(y\nu).
\end{equation*}
Function $T(\nu)$ is the Fourier transform of $H(m)$ and is given by $T(\nu) = e^{-2\pi |\nu|}$, hence $\sum_{m \in \mathbb{Z}} H \left( \frac{m}{y} \right) \ll y$. Note we have used the Poisson summation formula. Thus the contribution to $D$ is
\begin{equation}
\ll p^{a +na} y. 
\end{equation}
For the terms in $D$ that have $\chi \neq \psi$, even if $\chi$ and $\psi$ are primitive the product $\overline{\chi}\psi$ can be non-primitive because the conductors are not relatively prime. We have that for $g:\mathbb{Z}/q\mathbb{Z} \to \mathbb{C}$:
\begin{equation*}
\sum_{m=-\infty}^\infty g(m) f \left( \frac{m}{q} \right) = \sum_{b \text{ mod } q} g(b) F \left( \frac{b}{q} \right) = \sum_{\nu = -\infty}^\infty \hat{g}(-\nu) \hat{f} (\nu)
\end{equation*}
where $F(x) = \sum_{\nu=-\infty}^\infty \hat{f}(\nu) e^{-2\pi i \nu x}$. Applying this in our case, 
\begin{equation*}
\sum_{m=-\infty}^\infty \overline{\chi} \psi(m) H \left( \frac{m}{y} \right) = \frac{y}{p^a} \sum_{\nu = -\infty}^\infty \left( \sum_{b \text{ mod } p^a} \overline{\chi} \psi(b) e^{-2\pi i \nu b/p^a} \right) T \left( \frac{y\nu}{p^a} \right).
\end{equation*}
The interior sum is $\ll p^a$ since the number of characters is $\ll p^a$, and for $\nu=0$ it is zero since $\overline{\chi} \psi$ is non-trivial. Thus,
\begin{equation*}
\Bigl| \sum_{m=-\infty}^\infty \overline{\chi}\psi(m) H \left( \frac{m}{y} \right) \Bigr| \ll y \sum_{\nu \in \mathbb{Z}, \nu \neq 0} T \left( \frac{y\nu}{p^a} \right). 
\end{equation*}
Note this relation also holds for the computation performed by \cite{Ward}.

Assuming $\upsilon > \frac{1}{n}$ (which will be part of our constraint) gives that $y/p^a \to \infty$. We have
\begin{equation*}
\sum_{\nu \in \mathbb{Z}, \nu \neq 0} T \left( \frac{y\nu}{p^a} \right) \asymp \frac{2}{e^{2\pi y p^{-a}} - 1} \ll \frac{1}{y}.
\end{equation*}
Putting everything together, these terms of $D$ contribute
\begin{equation}
\ll p^{2a + na}. 
\end{equation}
Thus, we conclude that the two contributions for $\chi = \psi$ and $\chi \neq \psi$ combined give
\begin{equation}
\label{dd}
D \ll p^{a+na} y. 
\end{equation}
From \eqref{a2srwoD} and \eqref{dd}, even if $\pi$ is not tempered,
\ba
|A_{2,s/r}| &\ll& y^{\beta - 1/2 + \epsilon} p^{-a - an\beta} p^{a/2 + na/2} y^{1/2} \nn\\
\label{aa}
&\ll& y^{\beta + \epsilon} p^{-an\beta + a(n-1)/2}.
\ea

For $m \gg y^{1+\epsilon}$, $F_2 \left( \frac{m}{y} \right) \ll \frac{y^t}{m^t}$ for any integer $t\geq 1$, and applying Cauchy's inequality in \eqref{b2} gives 
\begin{equation*}
|B_{2, s/r}| \ll p^{-a} p^{-an\beta} y^t \left( \sum_{m \gg y^{1+\epsilon}} \frac{|a_{\tilde{\pi}}|^2}{m^{2-2\beta + 2t}} \right)^{1/2} D^{1/2}.
\end{equation*}
Splitting the sum over $m$ in dyadic intervals and using \eqref{avg},
\ba
|B_{2, s/r}| &\ll& p^{-a} p^{-an\beta} y^t \left( \sum_{i \gg (1+\epsilon) \log(y)} \frac{2^{i+\epsilon}}{2^{(i-1)(2-2\beta+2t)}} \right)^{1/2} D^{1/2} \nn \\
&\ll& p^{-a} p^{-an\beta} y^t \left( y^{2\beta - 2t -1 + \epsilon} \right)^{1/2} D^{1/2}.
\ea
Using the bound in \eqref{dd} gives
\begin{equation}
\label{bb}
|B_{2,s/r}| \ll y^{\beta + \epsilon} p^{-an\beta + a(n-1)/2}.
\end{equation}

From \eqref{ssum}, \eqref{aa} and \eqref{bb} we conclude that
\begin{equation}
\label{secondsum}
|S_{2,s/r}| \ll y^{\beta + \epsilon} p^{-an\beta + a(n-1)/2}.
\end{equation}
We want $S_{2,s/r} \to 0$ as $a\to \infty$. Taking $y=p^{an\upsilon}$ in \eqref{secondsum} gives the condition
\begin{equation*}
-\frac{1}{2} + n \left( \frac{1}{2} - \beta \right) + n \upsilon (\beta + \epsilon) < 0, 
\end{equation*}
which implies
\begin{equation}
\label{s2}
\upsilon < \frac{1 - n + 2n\beta}{2n(\beta + \epsilon)}.
\end{equation}

If $\pi$ is tempered then we need to check that $\upsilon$ satisfies conditions \eqref{s1_1} and \eqref{s2}. Thus, for a general $n$, taking $\epsilon \to 0$ the desired condition is
\begin{equation*}
1 - \frac{1}{n(1-\beta)} < \frac{1+n(2\beta -1)}{2n\beta},
\end{equation*}
or equivalently 
\begin{equation}
\beta > \frac{n-1}{n+1}.
\end{equation}

If $\pi$ is not tempered, then conditions \eqref{s1_2} and \eqref{s2} need to be satisfied. Thus, taking $\epsilon \to 0$ gives the condition
\begin{equation*}
1 - \frac{1}{2n(1-\beta)} < 1 + \frac{1-n}{2n\beta},
\end{equation*}
which implies 
\begin{equation*}
\beta > \frac{n-1}{n}.
\end{equation*}

\end{proof}

\begin{proof}[Proof of Corollary 1]
Take $s=r=1$ in Theorem 4 and use the functional equation. Note that if $\beta >1$, $L(\pi \otimes \chi, \beta)$ has an Euler product expansion and hence is nonvanishing.
\end{proof}

\section{Determination of GL(3) cusp forms}

Let $\pi \in \mathcal{A}_u(3)$ be an isobaric sum of unitary cuspidal automorphic representations of $\text{GL}(3,\mathbb{A}_\mathbb{Q})$. The local components $\pi_l$ are determined by the set of nonzero complex numbers $\{ \alpha_l, \beta_l, \gamma_l\}$, which we represent by the diagonal matrix $A_l(\pi)$. 

The $L$-factor of $\pi$ at a prime $l$ is given by 
\begin{equation}
L(\pi_l, s) = \text{det} (I- A_l (\pi) l^{-s})^{-1} = \prod_{j=1}^n (1- \alpha_l l^{-s})^{-1} (1- \beta_l l^{-s})^{-1} (1- \gamma_l l^{-s})^{-1}. 
\end{equation}
Let $S_0 =\{ l : \pi_l \text{ unramified and tempered}\}$, and let $S_1 =\{ l: \pi_l \text{ is ramified}\}$. Note that $S_1$ is finite. Take the union
\begin{equation*}
S = S_0 \cup S_1 \cup \{ \infty \}.
\end{equation*}
Since $\pi$ is unitary, $\pi_l$ is tempered iff $|\alpha_l|=|\beta_l|=|\gamma_l| =1$.

\begin{lem}
If $l \not\in S$ then 
\begin{equation}
A_l(\pi) =\{ u l^t, u l^{-t}, w \}.
\end{equation} 
with $|u|=|w|=1$. If $l \in S_0$ then 
\begin{equation*}
A_l (\pi) = \{ \alpha, \beta, \gamma \}
\end{equation*} 
with $|\alpha|=|\beta|=|\gamma|=1$. 
\end{lem}

\begin{proof}
Suppose first that $l \not\in S$. Then assume that $|\alpha_l| \neq 1$. Take $\alpha_l = u l^t$, with $|u|=1$ complex and $t \neq 0$ real. By unitarity, 
\begin{equation*}
\{ \overline{\alpha}_l, \overline{\beta}_l, \overline{\gamma}_l\} = \{ \alpha_l^{-1}, \beta_l^{-1}, \gamma_l^{-1}\}.
\end{equation*} 
Clearly $\overline{\alpha}_l \neq \alpha_l^{-1}$. Without loss of generality, take $\beta_l^{-1} = \overline{\alpha}_l$. Hence, this gives $\beta_l = u \cdot l^{-t}$. So, we must have $\overline{\gamma}_l = \gamma_l^{-1}$, hence $\gamma_l = w$ with $|w|=1$.  Hence
\begin{equation*}
A_l (\pi) = \{ u l^t, u l^{-t}, w\}
\end{equation*}
with $|u|=|w|=1$. 

Now suppose that $l \in S_0$. By unitarity $|\alpha_l| = |\beta_l| = |\gamma_l|=1$.
\end{proof}

\begin{proof}[Proof of Theorem 2]
Let $T = \{ l | \pi_l \text{ or } \pi'_l \text{ is ramified} \}$. This is a finite set. 

Consider $l \not\in T$ an arbitrary finite place with $l \neq p$. Let $A_l (\pi) =\{ \alpha_l, \beta_l, \gamma_l\}$ and $A_l (\pi') =\{ \alpha'_l, \beta'_l, \gamma'_l\}$. 

Applying Theorem 4, $a_\pi (n) = B^a C a_{\pi'} (n)$ for all $(n, p)=1$ and all but finitely many $a$. Since $a_\pi (1) = a_{\pi'} (1)$, then $B=C=1$. Thus, $a_\pi(l) = a_{\pi'} (l)$. 

We want to show that $A_l(\pi) = A_l(\pi')$. Indeed, 
\begin{equation}
\alpha_l + \beta_l + \gamma_l = \alpha'_l + \beta'_l + \gamma'_l
\end{equation}
and since $\pi$ and $\pi'$ have the same central character
\begin{equation}
\alpha_l \beta_l \gamma_l = \alpha'_l \beta'_l \gamma'_l. 
\end{equation}

To show that $\{\alpha_l\, \beta_l, \gamma_l\} = \{\alpha'_l, \beta'_l, \gamma'_l\}$, by Vieta's  formulas and the above two relations, it is enough to check that 
\begin{equation*}
\alpha_l \beta_l + \alpha_l \gamma_l + \beta_l \gamma_l = \alpha'_l \beta'_l + \alpha'_l \gamma'_l + \beta'_l \gamma'_l.
\end{equation*}

Suppose $A_l (\pi) = \{ul^t, ul^{-t}, w\}$ with $|u|=|w|=1$. Then
\begin{equation*}
\alpha_l \beta_l + \alpha_l \gamma_l + \beta_l \gamma_l = u^2 + uw(l^t + l^{-t}) = \frac{1}{u^2} + \frac{1}{uw} (l^t + l^{-t}) = \frac{w+u(l^t + l^{-t})}{u^2 w},
\end{equation*}
hence $\alpha_l \beta_l + \alpha_l \gamma_l + \beta_l \gamma_l = \frac{\alpha_l + \beta_l + \gamma_l}{\alpha_l \beta_l \gamma_l}$. 

Now suppose that $A_l (\pi) =\{ \alpha_l, \beta_l, \gamma_l\}$ with $|\alpha_l|=|\beta_l|=|\gamma_l|=1$. Then
\begin{equation*}
\alpha_l \beta_l + \alpha_l \gamma_l + \beta_l \gamma_l = \frac{1}{\alpha_l \beta_l} + \frac{1}{\alpha_l \gamma_l} + \frac{1}{\beta_l \gamma_l} = \frac{\alpha_l + \beta_l + \gamma_l}{\alpha_l \beta_l \gamma_l}.
\end{equation*}

Thus, this implies that whenever $\alpha_l + \beta_l + \gamma_l = \alpha'_l + \beta'_l + \gamma'_l$ and $\alpha_l \beta_l \gamma_l = \alpha'_l \beta'_l \gamma'_l$, we obtain that $\alpha_l \beta_l + \alpha_l \gamma_l + \beta_l \gamma_l = \alpha'_l \beta'_l + \alpha'_l \gamma'_l + \beta'_l \gamma'_l$.

We have thus shown that for $l \not\in T\cup \{p\} \cup \{ \infty\}$, $A_l (\pi) = A_l (\pi')$, hence $\pi_l \cong \pi'_l$. Since $T \cup \{p\} \cup \{\infty\}$ is a finite set, this implies that $\pi \cong \pi'$ by the Generalized Strong Multiplicity One Theorem. 
\end{proof}

Let $\pi$ be a unitary cuspidal automorphic representation of $\text{GL}(2, \mathbb{A}_\mathbb{Q})$ with $A_l (\pi) = \{ \alpha_l , \beta_l \}$. At an unramified place $l$, it has $a_l = \alpha_l + \beta_l$ and central character $\omega (\varpi_l) = \alpha_l \beta_l$, with $\varpi_l$ the uniformizer at $l$. There exists an isobaric automorphic representation $Ad(\pi)$ of $\text{GL}(3, \mathbb{A}_\mathbb{Q})$ (cf. \cite{GelJac}) such that at an unramified place $l$,
\begin{equation*}
a_l (Ad(\pi)) = \alpha_l /\beta_l + \beta_l / \alpha_l +1.
\end{equation*}

\begin{proof}[Proof of Theorem 3]
Theorem 2 implies that $Ad(\pi) \cong Ad(\pi')$. Then, by Theorem 4.1.2 in \cite{Ram}, we deduce that since $\pi$ and $\pi'$ have the same central character, there exists a quadratic character $\nu$ such that $\pi \cong \pi' \otimes \nu$.  
\end{proof}

\section{Complex adjoint $L$-functions}

We give a short overview of the theory of complex adjoint $L$-functions associated to an elliptic curve $E/\mathbb{Q}$. This section is mostly review and can be skipped by experts.

Let $E/\mathbb{Q}$ be an elliptic curve with conductor $N$ given by a global minimal Weierstrass equation over $\mathbb{Z}$:
\begin{equation}
\label{weier}
y^2 + a_1 xy + a_3 y = x^3 + a_2 x^2 + a_4 x + a_6. 
\end{equation}
Define the complex $L$-function of $E$ by the Euler product for $\text{Re}(s)> \frac{3}{2}$:
\begin{equation*}
L(E,s) = \prod_{r|N} \frac{1}{1-a_r r^{-s}} \prod_{r \not \ | N} \frac{1}{1-a_r r^{-s} + r^{1-2s}}
\end{equation*}
where $a_r = r+1 - \# E(\mathbb{F}_r)$ if $r \!\! \not| N$. If $r|N$ then $a_r$ depends on the reduction of $E$ at $r$ in the following way: $a_r =1$ if $E$ has split multiplicative reduction at $r$, $a_r = -1$ if $E$ has non-split multiplicative reduction at $r$ and $a_r =0$ if $E$ has additive reduction at $r$. In this paper, we will only consider primes $r$ such that $E$ has semistable reduction at $r$.

Let $f $ be the holomorphic newform of weight 2 and level $N$ associated to $E$. The Fourier coefficients $c_r$ of $f$ at $r \!\!\not| N$ prime coincide with the coefficients $a_r$ in the Euler product of $E$ and the $L$-function of $E$ is given by
\begin{equation*}
L(E,s) = \sum_{n=1}^\infty c_n n^{-s}.
\end{equation*}
If $\Lambda (E, s) = N^{s/2} (2\pi)^{-s}  \Gamma(s) L(E,s)$, then the following functional equation is satisfied:
\begin{equation*}
\Lambda (E,s) = \pm \Lambda (E, 2-s),
\end{equation*}
where the sign varies, depending on $E$. 
If we associate to $f$ a unitary cuspidal automorphic form $\pi$ of $\text{GL}(2, \mathbb{A}_{\mathbb{Q}})$ with trivial central character and conductor $N$ then we want to have $L(\pi, s)$ unitarily normalized by setting
\begin{equation*}
L_u(\pi, s) = (2\pi)^{-s -1/2} L\left(E, s+ \frac{1}{2} \right). 
\end{equation*}
By \cite{GelJac}, there exists $Sym^2(\pi)$ an isobaric representation of $\text{GL}(3, \mathbb{A}_\mathbb{Q})$. It is cuspidal only in the case when $E$ doesn't have complex multiplication. 
 
An elliptic curve $E$ over $\mathbb{Q}$ is of CM-type if $\text{End}(E) \otimes \mathbb{Q} = K$, with $K=\mathbb{Q}(\sqrt{-D})$ an imaginary quadratic number field. We have that $L(E,s) = L(\eta, s - 1/2)$ for some unitary Hecke character $\eta$ of the idele class group $C_K$. We can associate a newform $f$ of weight 2 and level $N$ such that $L(f,s) = L( \eta, s)$. Note that $N$ is the norm $N_{K/\mathbb{Q}}$ of the product of the different $\mathcal{D}_{K/\mathbb{Q}}$ and the conductor of $\eta$. The central character $\omega$ of $f$ is $\omega = \eta_0 \delta$, where $\eta_0$ is the restriction of $\eta$ to $\mathbb{Q}$ and $\delta$ the quadratic character associated to $K$. 

A representation induced by character $\eta$ as above is called dihedral. If $\pi = I_{K}^{\mathbb{Q}} (\eta)$ then 
\begin{equation*}
L(I_K^{\mathbb{Q}} (\eta), s) = L(\eta, s).
\end{equation*}

We can identify $\eta$ with a character of the Weil group $W_K$ by the isomorphism $C_K \cong W_K^{ab}$. There is a two-dimensional irreducible representation $\rho : W_\mathbb{Q} \to \text{GL}(V)$ such that 
\begin{equation*}
L(\pi, s) = L(\rho, s),
\end{equation*}
with $\rho$ induced from a character of $W_K$ and the associated $L$-function defined as in \cite{Tate}.  By the above identification, we write $\rho = I_{W_{K}}^{W_\mathbb{Q}} (\eta)$ (see also \cite{Krishnamurthy}).

Let $\pi$ be a (unitary) cuspidal automorphic representation of $\text{GL}(2, \mathbb{A}_{\mathbb{Q}})$. Suppose $\pi$ is dihedral, of the form $I_{K}^{\mathbb{Q}} (\eta)$ for a (unitary) character $\eta$ of $C_K$. Let $\tau$ be the non-trivial automorphism of the degree 2 extension $K/\mathbb{Q}$. Note that 
\begin{equation}
\label{inv}
\eta \eta^{\tau} = \eta_0 \circ N_{K/\mathbb{Q}},
\end{equation}
where $\eta_0$ is the restriction of $\eta$ to $\mathbb{Q}$.
We have, 
\begin{equation}
\label{isobaric}
I_K^{\mathbb{Q}} (\eta \eta^\tau) \cong \eta_0 \boxplus \eta_0 \delta
\end{equation}
where $\delta$ is the quadratic character of $\mathbb{Q}$ associated to $K/\mathbb{Q}$. 

If $\lambda, \mu$ are characters of $C_K$, then by applying Mackey:
\begin{equation}
\label{mackey}
I_K^{\mathbb{Q}} (\lambda) \boxtimes I_K^{\mathbb{Q}} (\mu) \cong I_K^{\mathbb{Q}}  (\lambda \mu) \boxplus I_K^{\mathbb{Q}} (\lambda \mu^\tau). 
\end{equation}
Taking $\lambda = \mu = \eta$ in \eqref{mackey} and using \eqref{inv} and \eqref{isobaric},
\begin{equation*}
\pi \boxtimes \pi \cong I_K^{\mathbb{Q}} (\eta^2) \boxplus \eta_0 \boxplus \eta_0 \delta. 
\end{equation*}
Since $\pi \boxtimes \pi = Sym^2 (\pi) \boxplus \omega$ with $\omega = \eta_0 \delta$,
\begin{equation}
Sym^2 (\pi) \cong I_K^{\mathbb{Q}} (\eta^2) \boxplus \eta_0. 
\end{equation}
Denote by $\pi'$ the cuspidal automorphic representation  $I_K^{\mathbb{Q}} (\eta^2)$ of $\text{GL}(2, \mathbb{A}_\mathbb{Q})$. We have
\begin{equation*}
L(Sym^2 \pi, s) = L(\pi', s) L(\eta_0, s). 
\end{equation*}
Twisting by some character $\chi$ gives
\begin{equation*}
L(Sym^2 \pi \otimes \chi, s) = L(\pi' \otimes \chi, s) L(\eta_0 \otimes \chi, s).
\end{equation*}
Note that $L(\pi' \otimes \chi, s) L(\eta_0 \otimes \chi, s)$ is entire unless $\chi = \eta_0^{-1}$ in which case
\begin{equation*}
L(Sym^2 \pi \otimes \eta_0^{-1}, s) = L(\pi' \otimes \eta_0^{-1}, s) \zeta (s)
\end{equation*}
has a pole at $s=1$. Hence, $L(Sym^2 \pi \otimes \chi, s)$ is entire for $\chi \neq \eta_0^{-1}$. 

More generally, we define the complex $L$-function associated to the symmetric square of an elliptic curve in the following way (cf. \cite{CoatesSch}). Let $l$ be an odd prime number. Take $E[l^n]$ to be the $l^n$-torsion and
\begin{equation*}
T_l (E) = \varprojlim E[l^n]
\end{equation*}
to be the $l$-adic Tate module of $E$. Consider the $V_l (E) = T_l (E) \otimes_{\mathbb{Z}_l} \mathbb{Q}_l$, which is 2-dimensional over $\mathbb{Q}_l$. There is a continuous natural action of $\text{Gal}(\overline{\mathbb{Q}}/\mathbb{Q})$ on $V_l$. Let $\Sigma_l (E) = Sym^2 H_l^1(E)$, where $H_l^1 (E) = \text{Hom}(V_l(E), \mathbb{Q}_l)$. Consider the representation
\begin{equation}
\label{l-adic}
\rho_l : \text{Gal}(\overline{\mathbb{Q}}/\mathbb{Q}) \to \text{Aut}(\Sigma_l (E)).
\end{equation}

The $L$-function of $Sym^2(E)$ is given by the Euler product 
\begin{equation}
L(Sym^2 E, s) = \prod_{r\ \text{prime}} P_r (r^{-s})^{-1}
\end{equation}
in the half-plane $\text{Re}(s) >2$. The polynomial $P_r (X)$ is 
\begin{equation}
P_r (X):= \text{det}(1-\rho_l(\text{Frob}_r^{-1})X |\Sigma_l (E)^{I_r}), \ l \neq r,
\end{equation}
with $I_r$ the inertia subgroup of $\text{Gal}(\overline{\mathbb{Q}_r}/\mathbb{Q}_r)$ and $\text{Frob}_r$ an arithmetic Frobenius element at $r$. By the N\'{e}ron-Ogg-Shafarevich criterion we have that
\begin{equation*}
P_r(X) = (1- \alpha_r^2 X)(1- \beta_r^2 X) (1-rX)
\end{equation*}
when $E$ has good reduction at $r$ (see \cite{CoatesSch}). The elements $\alpha_r$ and $\beta_r$ are the roots of the polynomial
\begin{equation*}
X^2 - a_r X + r
\end{equation*}
with $a_r$ the trace of Frobenius at $r$.

Let $L(Sym^2 E, \chi, s)$ denote the $L$-function associated to the twist of the $l$-adic representations by a Dirichlet character $\chi$. Note that $L(Sym^2 E, \chi, 1) =0$ for $\chi$ odd (cf. \cite{Dabro}). The critical points for $Sym^2 E$ are $s=1$ and $s=2$.

Let $\chi$ be a primitive even Dirichlet character with conductor $c_{\chi}$. Let $C$ denote the conductor of the $l$-adic representation \eqref{l-adic}. If $\tau(\chi)$ is the Gauss sum of character $\chi$, define
\begin{equation*}
W(\chi) = \chi(C) c_\chi^{1/2} \frac{\tau(\chi)}{\tau(\overline{\chi})^2}. 
\end{equation*}
Then, by Theorem 2.2 in \cite{CoatesSch}, which is based on results in \cite{GelJac}, if the conductor $N$ of $E$ satisfies $(c_\chi, N) =1$, then 
\begin{equation*}
\Lambda (Sym^2 E, \chi, s) = (C\cdot c_\chi^3)^{s/2} (2\pi)^{-s} \pi^{- \frac{s}{2}} \Gamma (s) \Gamma \left( \frac{s}{2} \right)  L(Sym^2 E, \chi, s)
\end{equation*}
has a holomorphic continuation over $\mathbb{C}$ and satisfies the functional equation
\begin{equation}
\label{functional}
\Lambda (Sym^2 E, \chi, s) = W(\chi) \Lambda(Sym^2 E, \overline{\chi}, 3-s).
\end{equation}

\section{Adjoint $p$-adic $L$-functions}

Fix $p$ an odd prime and let $E$ be an elliptic curve over $\mathbb{Q}$ with semistable reduction at $p$. We start this section by giving a constructing of a $p$-adic analogue to $L(Sym^2 E, s)$ by the Mellin transform of a $p$-adic measure $\mu_p$ on $\mathbb{Z}_p^\times$. We follow the approach in \cite{Dabro}.

Suppose first that $E$ has good reduction at $p$. Let $G_p = \text{Gal}(\mathbb{Q}(\zeta_{p^\infty})/\mathbb{Q})$ with $G_p \cong \mathbb{Z}_p^\times$ and let $X_p$ be the set of continuous characters of $G_p$ into $\mathbb{C}_p^\times$. For $\chi \in X_p$, let $p^{m_\chi}$ be the conductor of $\chi$. 

Consider the real and imaginary periods of a N\'{e}ron differential $\omega_E = \frac{\text{d}x}{2y+a_1 x +a_3}$ of a minimal Weierstrass equation for $E$ over $\mathbb{Z}$ such as \eqref{weier}, which we denote by $\Omega^{\pm}(E)$. Let 
\begin{equation*}
\Omega^{+} (Sym^2 E(1)):= (2\pi i)^{-1} \Omega^+ (E) \Omega^- (E) \text{ and } \Omega^+(Sym^2 E(2)):= 2\pi i \Omega^+ (E) \Omega^- (E)
\end{equation*}
be the periods for $Sym^2 E$ at the critical twists.

Note that since $\mathbb{Z}_p^\times \cong (1+p \mathbb{Z}_p) \times (\mathbb{Z}/p)^\times$, we can write $X:=X_p$ as the product of $X((\mathbb{Z}/p)^\times)$ with $X_0=X(1+p\mathbb{Z}_p)$. The elements of $X_0$ are called wild $p$-adic characters. Note that by Section 2.1 in \cite{Visik} we can give $X_0$ a $\mathbb{C}_p$-structure through the isomorphism of $X_0$ to the disk
\begin{equation} 
\label{padicdisk}
U:=\{ u \in \mathbb{C}_p^\times| | u -1 |<1\}
\end{equation}
constructed by mapping $\nu \in X_0$ to $\nu(1+p)$, with $1+p$ a topological generator of $1+p\mathbb{Z}_p$.

We follow the definition of the distribution $\mu_p(\Omega^{+} (Sym^2 E(\cdot)))$ on $G_p$ in \cite{Dabro}. Let $\chi \in X_0$ be a wild $p$-adic character, with conductor $p^{m_\chi}$ which can be identified with a primitive Dirichlet character. Then
\begin{equation}
\label{at1}
\int_{\mathbb{Z}_p^\times} \chi d\mu_p (\Omega^+(Sym^2 E(1))) := \alpha_p (E)^{-2m_\chi} \cdot \tau(\overline{\chi}) \cdot \frac{L(Sym^2 E, \chi, 1)}{\Omega^+ (Sym^2 E(1))}
\end{equation}
and 
\begin{equation}
\label{at2}
\int_{\mathbb{Z}_p^\times} \chi d\mu_p (\Omega^+ (Sym^2 E(2))):= \begin{cases} \alpha_p (E)^{-2m_\chi} \cdot \tau(\overline{\chi})^2 p^{m_\chi} \cdot \frac{L(Sym^2 E, \chi, 2)}{\Omega^+ (Sym^2 E(2))} & \chi \text{ even,} \\
0 & \chi \text{ odd}. \end{cases}
\end{equation}

By \cite{Dabro}, if $E$ has good ordinary reduction at $p$ then the distributions $\mu_p (\Omega^+ (Sym^2 E(\cdot)))$ are bounded measures on $G_p$. If $E$ has supersingular reduction at $p$ then the distributions $\mu_p(\Omega^+(Sym^2 E(\cdot)))$ give $h$-admissible measures on $G_p$, with $h=2$. Note that the set of $h$-admissible measures with $h=1$ is larger, but contains the bounded measures.

Now suppose that $E$ has bad multiplicative reduction at $p$ (either split or non-split). We define distributions $\mu_p(\Omega^{+}(Sym^2 E(\cdot)))$ on $G_p$ as in \cite{Dabro}. Let $\chi \in X_0$ denote a Dirichlet character of conductor $p^{m_\chi}$. Then
\begin{equation}
\label{at1bad}
\int_{\mathbb{Z}_p^\times} \chi d\mu_p (\Omega^+(Sym^2 E(1))) := \tau(\overline{\chi}) \cdot \frac{L(Sym^2 E, \chi, 1)}{\Omega^+ (Sym^2 E(1))}
\end{equation}
if $\chi$ is non-trivial, otherwise set the above integral to zero for the trivial character, and 
\begin{equation}
\label{at2bad}
\int_{\mathbb{Z}_p^\times} \chi d\mu_p (\Omega^+ (Sym^2 E(2))):= \begin{cases} \tau(\overline{\chi})^2 p^{m_\chi} \cdot \frac{L(Sym^2 E, \chi, 2)}{\Omega^+ (Sym^2 E(2))} & \chi \text{ even and non-trivial,} \\
0 & \chi \text{ odd or trivial}. \end{cases}
\end{equation}
By \cite{Dabro}, if $E$ has bad multiplicative reduction at $p$, then the distributions $\mu_p(\Omega^+(Sym^2 E(\cdot)))$ are bounded measures on $G_p$. 

Consider $\mu$ an $h$-admissible measure as above. Then
\begin{equation}
\label{Lmu}
\chi \to L_\mu (\chi):= \int_{\mathbb{Z}_p^\times} \chi d\mu
\end{equation}
is an analytic function of type $o(\log^h)$ (cf. \cite{Visik}). Note that for an analytic function $F$ to be of type $o(\log^h)$ it must satisfy
\begin{equation*}
\sup_{|u-1|_p <r} \| F(u)\| = o \left( \sup_{|u-1|_p < r} |\log_p^h (u)| \right) \text{ for } r\to 1_{-}.
\end{equation*}
An $h$-admissible measure $\mu$ is determined by the values $L_\mu(\chi x_p^r)$, where $\chi$ is a wild $p$-adic character and $x_p$ is the $p$-th cyclotomic character given by the action on the $p$-power roots of unity, with $r=0, 1, \dotsb h-1$. 

Consider the $p$-adic distribution $\mu= \mu_p(\Omega^+(Sym^2 E(2)))$ as defined above. Denote by $L_p$ the corresponding $p$-adic $L$-function. We have
\begin{equation*}
L_p(Sym^2 E, \chi, s) := \int_{\mathbb{Z}_p^\times} \chi(x) \langle x \rangle^s d\mu. 
\end{equation*}
where $\langle \cdot \rangle : \mathbb{Z}_p^\times \to 1 + p\mathbb{Z}_p$, $\langle x \rangle = \frac{x}{\omega(x)}$, with $\omega: \mathbb{Z}_p^\times \to \mathbb{Z}_p^\times$ the Teichm\"{u}ller character.  

Let's first prove the following lemma:

\begin{lem}
Let $p$ be an odd prime. Let $E, E'$ be elliptic curves over $\mathbb{Q}$ with semistable reduction at $p$ such that $L_p(Sym^2 E, n) = C L_p (Sym^2 E', n)$, for an infinite number of integers $n$ prime to $p$ in some set $Y$, and some constant $C \in \overline{\mathbb{Q}}$. Then for every finite order wild $p$-adic character $\chi$,
\begin{equation*}
L_p (Sym^2 E, \chi, s) = C L_p (Sym^2 E', \chi, s)
\end{equation*}
holds for all $s \in \mathbb{Z}_p$. 
\end{lem}

\begin{proof}
We follow the approach in \cite{LuoRam}. 
Let 
\begin{equation*}
G(\nu) = L_p (Sym^2 E, \nu) - CL_p(Sym^2 E', \nu)
\end{equation*}
for every $\nu \in X_0$. $G$ vanishes on $X_1 =\{\alpha_n = \langle x \rangle^n| n\in Y\}$ by hypothesis; we want to show that $G$ vanishes on $X_0$. We will use the fact that $G$ is an analytic function on $X_0$ of type $o(\text{log}^h)$ (as in \eqref{Lmu}). $G$ considered as an analytic function on $U$ (see \eqref{padicdisk}) vanishes on the subset 
\begin{equation*}
U_1 =\{ (1+p)^n| n \in Y\}. 
\end{equation*}

We will show that there exists $r = 1/p$ such that the number of zeros $z$ of $G$ such that $|z-1|=r$ is infinite. Indeed, for all $n \in Y$ elements in an infinite set with $n$ relatively prime to $p$ as above, $z_n:= (1+p)^n \in U_1$ is a zero of $G$ and
\begin{equation*}
|z_n -1 | = |(1+p)^n - 1|_p = \Bigl| \sum_{j=1}^n {n \choose j} p^j \Bigr|_p = \frac{1}{p}.
\end{equation*}
By Section 2.5 in \cite{Visik}, $G$ is identically zero on $U$. 
\end{proof}

\begin{proof}[Proof of Theorem 1]
By Lemma 4, for every finite order wild $p$-power character $\chi$, the identity
\begin{equation}
L_p(Sym^2 E, \chi, s) = CL_p (Sym^2 E', \chi, s)
\end{equation}
holds for all $s \in \mathbb{Z}_p$. By equation \eqref{at2},  if $E$ has good reduction at $p$ then
\begin{equation}
\label{ifgood}
\alpha_p(E)^{-2m_\chi} L(Sym^2 E, \chi, 2) = C' \alpha_p(E')^{-2m_\chi} L(Sym^2 E', \chi, 2) 
\end{equation}
for some $C' \in  \overline{\mathbb{Q}}$. 

If $E$ has bad multiplicative reduction at $p$, then by \eqref{at2bad}, 
\begin{equation}
\label{ifbad}
L(Sym^2 E, \chi, 2) = C'' L(Sym^2 E', \chi, 2)
\end{equation}
for some $C'' \in \overline{\mathbb{Q}}$.

Let $\pi, \pi'$ be the isobaric sums of unitary cuspidal automorphic representations over $\text{GL}(3, \mathbb{A}_{\mathbb{Q}})$ associated to $Sym^2 E$ and $Sym^2 E'$ respectively (If E and E' are not CM, then these are just unitary cuspidal representations). Then the unitarized $L$-functions $L_u$ corresponding to $\pi$ and $\pi'$ satisfy $L_u(\pi, s) = L(Sym^2 E, s+1)$. Hence, it follows that if $E$ has semistable reduction at $p$, we have from \eqref{ifgood} and \eqref{ifbad} that
\begin{equation*}
L(\pi \otimes \chi, 1) = C_1 C_2^{m_\chi} L(\pi' \otimes \chi, 1)
\end{equation*}
for all wild $p$-power characters $\chi$ of conductor $p^{m_\chi}$ with $m_\chi$ sufficiently large and by the discussion in Section 5, the twisted $L$-functions are entire. Then by Theorem 2, we conclude that $\pi \cong \pi'$ and thus $\text{Ad}(\eta) \cong \text{Ad}(\eta')$ where $\eta, \eta'$ are the unitary cuspidal automorphic representations of $\text{GL}(2, \mathbb{Q})$ associated to $E$. By Theorem 4.1.2 in \cite{Ram} we conclude that $\eta' = \eta \otimes \nu$ with $\nu$ a quadratic character  since $\omega_\eta = \omega_{\eta'}=1$. Write $\nu(\cdot) = \left(\frac{\cdot}{D}\right)$. It then follows by Faltings' isogeny theorem that $E'$ is isogenous to $E_D$, where for the elliptic curve $E$ given by the equation $y^2=f(x)$ we have that $E_D$ is given by the equation $Dy^2 = f(x)$. Clearly if the conductors of $E$ and $E'$ are square free, then $E\approx E'$.

\end{proof}

\begin{remark}
Suppose $E, E'$ are CM elliptic curves and let $\eta, \eta'$ be their associated idele class characters over the imaginary quadratic number fields $K$ and $K'$ respectively. If we let $\pi, \pi'$ be the representations induced by the characters $\eta, \eta'$, then they are dihedral. By the discussion in Section 5, 
\begin{equation}
L(Sym^2 \pi, s) = L\left( I_{K}^\mathbb{Q} (\eta^2), s \right) L(\eta_0, s)
\end{equation}
where $\eta_0$ denotes the restriction of $\eta$ to $\mathbb{Q}$, and similarly for $\pi'$. If $K=K'$ then $\eta_0 = \eta_0'$. Hence, Theorem 1 for $E, E'$ as above is a consequence of Lemma 4 and Theorem A in \cite{LuoRam}, since  $I_{K}^\mathbb{Q} (\eta^2)$ is a cuspidal automorphic representation of $\text{GL}(2, \mathbb{A}_\mathbb{Q})$. 

It is unclear if for $K\neq K'$ Theorem 1 can be reduced to a consequence of a result on the determination of $\text{GL}(2)$ cusp forms. The special values $L(\eta_0 \otimes \chi, 1)$ and $L(\eta_0 \otimes \chi, 2)$ can be expressed in terms of the generalized Bernoulli numbers $B_{1, \overline{\eta_0 \chi}}$ and $B_{2, \overline{\eta_0 \chi}}$ respectively, but there is no clear way to separate the contributions from $\eta_0$ and $\chi$ in $L(\eta_0 \otimes \chi, 1)$. 

\end{remark}

\end{document}